\documentclass[12pt,reqno]{amsart}

\usepackage{graphicx,subfigure}
\usepackage{hyperref}
\hypersetup{colorlinks=true, citecolor=blue, linkcolor=red}

\numberwithin{equation}{section} \numberwithin{figure}{section}
\numberwithin{table}{section} 
{ \theoremstyle{plain}

}

{ \theoremstyle{definition}
\newtheorem{thm}{Theorem}

\newtheorem{lem}{Lemma}

\newtheorem{rem}{Remark}

\numberwithin{equation}{section} \numberwithin{lem}{section}
\numberwithin{thm}{section} \numberwithin{cor}{section}
\numberwithin{pro}{section} \numberwithin{rem}{section}
}

\begin{document}

\title[Confinement of bounded entire solutions to elliptic systems]{On the confinement of bounded entire solutions to a class of semilinear elliptic systems}



\author{Christos Sourdis} \address{Department of Mathematics and Applied Mathematics, University of
Crete.}
              \email{csourdis@tem.uoc.gr}           





\maketitle

\begin{abstract}
Under appropriate assumptions, we show that all bounded entire
solutions to a class of semilinear elliptic systems are confined in
a convex domain. Moreover, we prove a Liouville type theorem in the
case where the domain is strictly convex. Our result represents an
extension, under less regularity assumptions, of a recent result in
\cite{smyrnelis}. We also provide several applications.
\end{abstract}

\section{Introduction and statement of the main result}

The following result is contained in the very recent paper of P.
Smyrnelis \cite{smyrnelis}:
\begin{thm}\label{thmsmyrnelis}
Let $W\in C^{2,\alpha}(\mathbb{R}^m,\mathbb{R})$, $\alpha\in (0,1)$,
be such that
\[
u\cdot \nabla W(u)>0\ \ \textrm{for}\ \ u\in \mathbb{R}^m\
\textrm{with}\ |u|>R,
\]
where $R>0$ is some constant. If $u\in
C^2(\mathbb{R}^n;\mathbb{R}^m)\cap
L^\infty(\mathbb{R}^n;\mathbb{R}^m)$ is an entire solution to the
equation
\begin{equation}\label{eqEq}
\Delta u =\nabla W(u),\ \ x\in \mathbb{R}^n,
\end{equation}
we have that $|u(x)|\leq R$, $x\in \mathbb{R}^n$. In addition, if
$u$ is not constant, then $|u(x)|<R$, $x\in \mathbb{R}^n$.
\end{thm}

The proof is based on the $P$-function technique, see \cite{sperb}
for the case of scalar equations. Essentially, this technique
consists in applying the maximum principle to a  second order
elliptic equation that is satisfied by a convenient scalar function
$P(u;x)$ where $u$ solves (\ref{eqEq}). The choice made in
\cite{smyrnelis} was
\[
P(u;x)=\frac{1}{2}|\nabla u(x)|^2+C\left(\left|u(x)\right|^2-R^2
\right),
\]
for some large constant $C>0$. In fact, the gradient structure of
the righthand side of (\ref{eqEq}) did not play any role in the
proof of the above theorem; in this regard, see \cite[Thm. 2.4]{alamaPeli}.
 We point out that the reason for
assuming that $W\in C^{2,\alpha}$ was to justify taking the
Laplacian of the above function $P$.

The purpose of this note is to prove the following extension and
improvement (as far as regularity is concerned) of the above result,
and present some applications.

%


\begin{thm}\label{thmmine}
Let $\mathcal{D}$ be a smooth convex domain of $\mathbb{R}^m$ (at
least $C^2$). We assume that $F\in C^{0,1}(\mathbb{R}^m;\mathbb{R})$
and
\begin{equation}\label{eqass}
(u-u_0)  \cdot F(u)>0\ \ \forall\ u\in \mathbb{R}^m\setminus
\bar{\mathcal{D}},
\end{equation}
where  $u_0\in \partial \mathcal{D}$ is such that
$|u-u_0|=\textrm{dist}(u,\partial \mathcal{D})$.

 If $u\in
C^2(\mathbb{R}^n;\mathbb{R}^m)\cap L^\infty
(\mathbb{R}^n;\mathbb{R}^m)$ is an entire solution of
\begin{equation}\label{eqEqnew} \Delta u=F(u)\ \ \textrm{in}\ \
\mathbb{R}^n,
\end{equation}
then
\begin{equation}\label{eqassert1}
u(x)\in \bar{\mathcal{D}},\ \ x\in \mathbb{R}^n.
\end{equation}
In addition, if $\mathcal{D}$ is strictly convex, and $u$ is non
constant, we have that
\[
u(x)\in {\mathcal{D}},\ \ x\in \mathbb{R}^n.
\]
\end{thm}

\begin{rem}
Part of the above theorem as well as part of the result in
Subsection \ref{subAllen} appeared in a remark in a subsequent
version of \cite{smyrnelis} which, however, appeared after the
current paper was posted on the Arxiv.
\end{rem}

\section{Proof of the main result}
The first assertion of Theorem \ref{thmmine} will follow from the
following lemma which is of independent interest.

\begin{lem}\label{lemmine}
Let $F\in C^{0,1}(\mathbb{R}^m;\mathbb{R}^m)$ satisfy
\begin{equation}\label{eqassnew}
F(u_1,u_2,\cdots,u_m)\cdot (u_1,0,\cdots,0)>0 \ \ \textrm{if}\
 u_1>L,\ u_i\in \mathbb{R},\ i=2,\cdots,m,
\end{equation}
for some $L\geq 0$.

If $u=(u_1,\cdots,u_m)\in C^2(\mathbb{R}^n;\mathbb{R}^m)$ is an
entire bounded solution to the elliptic system (\ref{eqEqnew}), we
have that
\[
 u_1(x) \leq L,\ \ x\in \mathbb{R}^n.
\]
\end{lem}

\begin{proof}
We will  argue by contradiction. 
For this purpose, suppose that
\begin{equation}\label{eqcontranew}M=\sup_{x\in
\mathbb{R}^n}u_1(x)>L,
\end{equation}
(clearly $M<\infty$). There exist $x_j\in \mathbb{R}^n$ such that
\[
u_1(x_j)\to M.
\]
Let
\[
v_j(x)=u(x+x_j).
\]
We have that
\[
\Delta v_j=F(v_j), \ \ |v_j|\leq C_1, \ \ x\in\mathbb{R}^n,\ \ j\geq
1,
\]
for some $C_1>0$. Moreover, the first component of  $v_j$ satisfies
\[
\left(v_j\right)_1(0)=u_1(x_j)\to M\ \ \textrm{and}\ \
\left(v_j\right)_1(x)\leq M,\ \ x\in \mathbb{R}^n.
\]
 By standard interior elliptic regularity estimates
\cite{evans,Gilbarg-Trudinger}, we deduce that
\begin{equation}\label{eqgrad}
\| v_j\|_{C^{2,\alpha}(\mathbb{R}^n;\mathbb{R}^m)}\leq C_2,\ \ j\geq
1,
\end{equation}
where $0<\alpha<1$ is fixed, for some $C_2>0$.
 Hence, by  well known compactness imbeddings (see \cite{evans,Gilbarg-Trudinger}), and the standard
diagonal Cantor type argument, passing to a subsequence if
necessary, we find that
\[
v_j\to V \ \ \textrm{in}\ \ C^2_{loc}(\mathbb{R}^n;\mathbb{R}^m),
\]
for some $V=(V_1,\cdots,V_m)$ which satisfies
\begin{equation}\label{eqcontra2new}
\Delta V=F(V)\ \ \textrm{in}\ \ \mathbb{R}^n, \ \ \textrm{and}\ \
V_1(0)=\sup_{x\in \mathbb{R}^n}V_1(x)=M.
\end{equation}
In view of (\ref{eqassnew}) and (\ref{eqcontranew}), we may assume
that
\[
\Delta V_1>0,\ \ |x|<\delta,
\]
for some small $\delta>0$. On the other hand, the second relation in
(\ref{eqcontra2new})  contradicts the strong maximum principle (see
\cite{evans,Gilbarg-Trudinger}).
\end{proof}

We can now proceed to the proof of our main result.

\begin{proof}[Proof of Theorem \ref{thmmine}]
Let $p\in \partial D$ and $T_p$ denote the tangent plane to
$\partial D$ at $p$. That tangent plane separates $\mathbb{R}^m$ to
two open connected components. The one component contains
$\mathcal{D}$ and the other one contains $\mathbb{R}^m \setminus
\bar{\mathcal{D}}$. The first assertion of the theorem will follow
if we show that the points $u(x)$  belong to the closure of the
component that contains $\mathcal{D}$ for every $x\in \mathbb{R}^n$.
Since the equation (\ref{eqEqnew}) is invariant under
 translations and rotations, we may assume that $p$ is the origin
 and that $T_p$ is the hyperplane $\{u_1=0 \}$ with $\mathcal{D}\subset \{u_1<0 \}$.
Clearly, assumption (\ref{eqassnew}) is satisfied with $L=0$. It
then follows from Lemma \ref{lemmine} that the first component of
$u$ satisfies is nonnegative, as desired.

The second assertion of the theorem  follows directly from
\cite{weinberger}, something which was not noticed in \cite{smyrnelis}. For the sake of completeness, we will give a
self-contained proof in the spirit of \cite{wang}. Let $u\in
C^2(\mathbb{R}^n;\mathbb{R}^m)$ be a solution to (\ref{eqEqnew})
such that (\ref{eqassert1}) holds and $u(x_0)\in
\partial \mathcal{D}$ for some $x_0 \in \mathbb{R}^n$, where
$\mathcal{D}$  is additionally assumed to be strictly convex.
 We denote the
signed distance of a point $u\in \mathbb{R}^m$ from $\partial
\mathcal{D}$ by $d(u)$, that is $d(u)<0$ if $u\in \mathcal{D}$ and
$d(u)>0$ if $u\in \mathbb{R}^m\setminus \bar{\mathcal{D}}$. It is
well known that the function $d$ is convex in $\mathbb{R}^m$, and
smooth in a tubular neighborhood of $\partial \mathcal{D}$ (see
\cite{Gilbarg-Trudinger}). In particular, by the strict convexity of
$\partial \mathcal{D}$, we have that the Hessian
\begin{equation}\label{eqpsar}
\partial^2 d(u)\ \ \textrm{is positive definite for}\ u\in
\partial \mathcal{D}.
\end{equation}
The function
\[
U(x)=d\left(u(x) \right)
\]
is smooth in a neighborhood of $x_0$, say if $|x-x_0|<\epsilon$ for
some small $\epsilon>0$. For such $x$, using (\ref{eqass}),
(\ref{eqEqnew}) and (\ref{eqpsar}), we find that
\begin{equation}\label{eqcomplex}\begin{array}{rcl}
   \Delta U(x) & = & \textrm{tr} \left\{\left(\partial^2 d\left(u(x)\right)\right)\left(\nabla u(x)\right)\left(\nabla
u(x)\right)^T\right\}+ \left[(\nabla d)\left(u(x)
\right)\right]\cdot F\left(u(x) \right) .
 \\
      &  &  \\
      & \geq &  c |\nabla u(x)|^2+ \left[(\nabla d)\left(u(x)
\right)\right]\cdot F\left(u(x) \right),
  \end{array}
\end{equation}
for some $c>0$, having decreased $\epsilon>0$ if needed.
 For
$|x-x_0|<\epsilon$, let
\[
Q(x)=\left\{\begin{array}{cc}
              \frac{\left[(\nabla d)\left(u(x)
\right)\right] \cdot F\left(u(x) \right)}{U(x)}, &\textrm{if}\ U(x)<0, \\
                &   \\
              0, & \textrm{otherwise}.
            \end{array}
 \right.
\]
If $u(x)\in \mathcal{D}$ with $|x-x_0|<\epsilon$, let $\tilde{u}\in
\partial \mathcal{D}$ be such that $U(x)=-\left|u(x)-\tilde{u}
\right|$. Note that $\nabla d(\tilde{u})=\nu_{\tilde{u}}$, where
$\nu_{\tilde{u}}$ denotes the outward unit normal vector to
$\partial \mathcal{D}$ at $\tilde{u}$. So, from (\ref{eqass}), we
have that
\[
Q(x)\leq \frac{\left[(\nabla d)\left(u(x) \right)\right] \cdot
F\left(u(x) \right)-\nabla d(\tilde{u}) \cdot F\left(\tilde{u}
\right)}{-\left|u(x)-\tilde{u} \right|}\leq C_3,
\]
for some constant $C_3>0$, where we used the Lipschitz continuity of
$F$ and the smoothness of $\partial \mathcal{D}$. Since
\[
\Delta U\geq Q(x)U\ \ \textrm{if}\ \ |x-x_0|<\epsilon,\ \
\textrm{and}\ \ U(x)\leq 0=U(x_0),
\]
a refinement of Hopf's boundary point lemma (see \cite[Ch.
9]{evans})
 yields that $\nabla U(x_0)\neq 0$ or $U$ is constant, namely zero, for
 $|x-x_0|<\epsilon$ (apply the aforementioned lemma in \cite{evans} for $v=-U\geq 0$
 and $c=Q$, noting that $c$ bounded from above suffices for the proof to go through).
On the other hand, since $U(x)\leq 0=U(x_0)$ for $|x-x_0|<\epsilon$,
we have that $\nabla U(x_0)=0$.
  Thus, only the second scenario is possible.
 Hence, we infer that
 \[
U(x)=0 \ \ \textrm{if}\ \ |x-x_0|<\epsilon. \] By a simple
continuity argument, we have that
\[
U(x)=0, \ \ x\in \mathbb{R}^n,\ \ \textrm{that is}\ \ u(x)\in
\partial \mathcal{D},\ \ x\in \mathbb{R}^n.
\]
Observe that this holds for $\mathcal{D}$ smooth and merely convex.
Now, we will make use of the strict inequality in (\ref{eqpsar}). By
differentiating the above relation, and making use of
(\ref{eqcomplex}), we infer that
\[
\nabla u(x)=0,\ \ x\in \mathbb{R}^n,
\]
that is
\[
u(x)=u(x_0),\ \ x\in \mathbb{R}^n,
\]
as desired.
\end{proof}
\section{Applications}
Below, we will present some applications of our main result.

\subsection{The Ginzburg-Landau system}
Consider the Ginzburg-Landau system which arises in
superconductivity:
\[
A \Delta u=\left( |u|^2-1\right)u,\ \ x\in \mathbb{R}^n,
\]
where $u$ takes values in $\mathbb{R}^m$ and $A$ is a diagonal
matrix with positive entries in the diagonal (see for example
\cite[pg. 210]{smoller}). In the case where $A$ is the identity, it was shown
in \cite{smyrnelis}, as a corollary of Theorem \ref{thmsmyrnelis},
that every entire bounded solution satisfies $|u|\leq 1$ in
$\mathbb{R}^n$, and $|u|<1$ in $\mathbb{R}^n$ if $u$ is nonconstant
(actually, it was already shown in \cite{farinaDiffInt} that every
entire solution is bounded and satisfies $|u|\leq 1$ in
$\mathbb{R}^n$). In the general case, where $A$ is not a positive
constant multiple of the identity, it follows readily from Theorem
\ref{thmmine} that the same properties continue to hold (actually, this was also proven in a more general result in \cite{smyrnelis} that is in the spirit of Theorem \ref{thmsmyrnelis}). Indeed,
firstly observe that the function $v=Au$ satisfies
\[
\Delta v=\left(|A^{-1}v|^2-1 \right)A^{-1}v\ \ \textrm{in}\ \
\mathbb{R}^n.
\]
Let $\mathcal{D}$ be the smooth and strictly convex domain
$\left\{v\in \mathbb{R}^m\ :\ |A^{-1}v|<1 \right\}$. Let $v\in
\mathbb{R}^m\setminus \bar{D}$, that is $|A^{-1}v|>1$, and $v_0\in
\partial \mathcal{D}$ be such that $|v-v_0|=\textrm{dist}(v,\partial
\mathcal{D})$. Since the outer unit normal vector to $\partial
\mathcal{D}$ at $v_0$ is $\frac{A^{-2}v_0}{|A^{-2}v_0|}$, we have
that
\[
v=v_0+\frac{|v-v_0|}{|A^{-2}v_0|}A^{-2}v_0.
\]
Using this, we find readily that
\[\begin{array}{rcl}
  \left(|A^{-1}v|^2-1 \right)A^{-1}v\cdot (v-v_0) &= & \left(|A^{-1}v|^2-1 \right)
\frac{|v-v_0|}{|A^{-2}v_0|}\left(A^{-1}v\right)\cdot
\left(A+\frac{|v-v_0|}{|A^{-2}v_0|}
A^{-1} \right)^{-1}\left(A^{-1}v\right) \\
      &   &   \\
      &  \geq  & c|A^{-1}v|^2>c,
  \end{array}
\]
for some positive $c$. Theorem \ref{thmmine} then implies that
$|A^{-1}v|\leq 1$ in $\mathbb{R}^n$, and $|A^{-1}v|< 1$ in
$\mathbb{R}^n$ if $v$ is nonconstant. The corresponding assertions
for $u=A^{-1}v$ follow at once.
\subsection{The vectorial Allen-Cahn equation}\label{subAllen}
Let $W:\mathbb{R}^2\to \mathbb{R}$ be a smooth function with three
global nondegenerate  minima at $a,b,c\in \mathbb{R}^2$ (not
contained in the same line). Some special bounded solutions $u\in
C^2(\mathbb{R}^2;\mathbb{R}^2)$ of the elliptic system
(\ref{eqEq}) with $n=2$, taking values close to $a$, $b$ or $c$
away from three half-lines (domain walls) that meet at the origin,
are related to the study of some models of three-boundary motion
in material science (see \cite{saez} and the references therein).
The most natural choice is
\[
W(u)=|u-a|^2|u-b|^2|u-c|^2.
\]
Let $u$ be a bounded entire solution to (\ref{eqEq}) for this $W$.
By translating and rotating this solution, we may assume that the
resulting function $\tilde{u}$ solves (\ref{eqEq}) with $W$ as above
but with $a=(0,a_2)$,  $b=(0,-a_2)$, $c=(c_1,c_2)$ such that $a_2>0$
and $c_1<0$. It is easy to show that (\ref{eqassnew}) is satisfied.
Hence, by Lemma \ref{lemmine}, we see that the first component of
$\tilde{u}$ is non-positive. In turn, reversing the Euclidean
motions, this implies that the values of $u$ are on the same side of
the line joining $a$ and $b$ as the triangle $\widehat{abc}$.
Analogously, we can show that the range of $u$ is contained in the
closed $\widehat{abc}$ triangle. In fact, from the proof of the
second assertion of this theorem, we find that if a bounded entire
solution touches one of the sides of the triangle, then it must be
contained in this side for all $x\in \mathbb{R}^n$; clearly, this
cannot happen for the solutions constructed in \cite{saez} which
``take'' all three phases.

\subsection{Symmetry of components of a semilinear elliptic system}
Our Lemma \ref{lemmine} also implies the following interesting
property: If $F\in C^{0,1}(\mathbb{R}^2;\mathbb{R}^2)$ satisfies
\[
(-u_2,u_1)\cdot F(u_1,u_2)>0\ \ \textrm{for}\ \ u_1\neq u_2,
\]
then every bounded entire solution $u=(u_1,u_2)$ of (\ref{eqEqnew})
satisfies
\[
u_1(x)=u_2(x),\ \ x\in \mathbb{R}^n.
\]

\subsection{Domain walls in the coupled Gross-Pitaevskii equations}
In \cite{alamaPeli}, the authors studied solutions to the system
\begin{equation}\label{eqGP}\begin{array}{l}
  u_1''=g_{11}(u_1^2-a^2)u_1+g_{12}u_1 u_2^2, \\
    \\
  u_2''=g_{22}(u_2^2-b^2)u_2+g_{12}u_1^2 u_2, \\
\end{array}
\end{equation}
with
\begin{equation}\label{eqGPbdry}
\left(u_1(x),u_2(x)\right)\to (a,0)\ \textrm{as}\ x\to \infty;\ \
\left(u_1(x),u_2(x)\right)\to (0,b)\ \textrm{as}\ x\to -\infty,
\end{equation}
where
\begin{equation}\label{eqSegr---}
a=\frac{\sqrt{\mu}}{\sqrt[4]{g_{11}}}, \ \
b=\frac{\sqrt{\mu}}{\sqrt[4]{g_{22}}},
\end{equation}
for some $\mu>0$ such that
\begin{equation}\label{eqSegre}
g_{12}>\sqrt{g_{11}g_{22}}.
\end{equation}
This heteroclinic connection problem arises in the study of domain
wall solutions in coupled Gross-Pitaevskii equations on the real
line. Among the many results, in Theorem 2.4 they showed that
solutions of (\ref{eqGP})--(\ref{eqGPbdry}) satisfy
\[
\frac{u_1^2(x)}{a^2}+\frac{u_2^2(x)}{b^2}\leq 1, \ \ x\in
\mathbb{R}.
\]
Their approach was based on a rather ad-hoc argument in the spirit
of the P-function method that we discussed in the introduction. As
an application of our Theorem \ref{thmmine} we can provide a
simpler proof  which, in fact, holds for any entire, bounded
solution to the corresponding elliptic system to (\ref{eqGP}).

Lets take a point $(u_1,u_2)$ outside of the above ellipse (which
is clearly convex), that is
\begin{equation}\label{eqellipseOUT}
\frac{u_1^2}{a^2}+\frac{u_2^2}{b^2}>1.
\end{equation}
By the symmetry of the system, we may assume without loss of
generality that $u_1,u_2\geq 0$. Let $(u_1^0,u_2^0)$ be its
closest point on the ellipse. The  latter point can be given
explicitly (see \cite[pg. 54]{convex}) but this will not be
needed. The only thing that we will use is that the vector
$(u_1-u_1^0,u_2-u_2^0)$ has nonnegative components and at least
one which is positive. To conclude, we note that, for such
$(u_1,u_2)$, the corresponding component of the system
(\ref{eqGP}) is positive as well. Really, assume that $u_1>0$.
Then, thanks to (\ref{eqSegr---}), (\ref{eqSegre}) and
(\ref{eqellipseOUT}), we find that
\[\begin{array}{ccc}
  g_{11}(u_1^2-a^2)u_1+g_{12}u_1 u_2^2 & > & -g_{11}\frac{a^2}{b^2}u_2^2u_1+g_{12}u_1 u_2^2 \\
    &   &   \\
    &  = &  (g_{12}-\sqrt{g_{11}g_{22}})u_1u_2^2\geq 0. \\
\end{array}
\]
Analogously we argue for the second equation.

\section*{Acknowledgements} This research was supported by the
ARISTEIA (Excellence) programme ``Analysis of discrete, kinetic
and continuum models for elastic and viscoelastic response'' of
the Greek Secretariat of Research.

\end{document}